\documentclass{conm-p-l}

\usepackage{amsmath,amssymb,amscd,amsxtra,dsfont}
\usepackage{graphicx,pst-grad,pst-plot,pst-coil}
\usepackage{pstricks}
\usepackage{eucal}
\usepackage{lineno}

\newtheorem{theorem}{Theorem}[section]
\newtheorem{lemma}[theorem]{Lemma}

\theoremstyle{definition}
\newtheorem{definition}[theorem]{Definition}

\theoremstyle{remark}

\numberwithin{equation}{section}

\newtheorem{proposition}[theorem]{Proposition}

\usepackage[all]{xy}
\def\bfig{\vcenter\bgroup}
\def\efig{\egroup}

\renewenvironment{proof}{
{\medskip\par\noindent\it Proof.\ }}{\vskip 2ex\par}

\newcommand{\C}{\mathbb{C}}

\newcommand{\ga}{\alpha}
\newcommand{\gb}{\beta}

\newcommand{\R}{\mathbb{R}}

\newcommand{\comp}{\mathrel{\scriptstyle\circ}}

\def\dual{^{\scriptstyle\vee}}

\newcommand{\cinf}{{C^{\infty}}}

\newcommand{\diag}{\operatorname{diag}}

\newcommand{\Hom}{\operatorname{Hom}}

\newcommand{\hu}{\underline{h}}

\newcommand{\ju}{\underline{j}}

\newcommand{\term}[1]{\textbf{\textit{#1}}}

\newcommand{\tq}{\tilde{q}}

\newcommand{\ty}{\tilde{y}}

\newcommand{\Unit}{\text{U}}

\usepackage{color}
\newcommand{\tred}[1]{{\color{black}{#1}}}

\begin{document}
\title
{Gysin Formulas and Equivariant Cohomology}

\author{Loring W. Tu}
\address{Department of Mathematics\\
Tufts University\\
Medford, MA 02155} 
\email{loring.tu@tufts.edu}
\urladdr{ltu.pages.tufts.edu}
\date{}

\begin{abstract}
Under Poincar\'e duality, a smooth map of compact oriented manifolds
induces a pushforward map in cohomology, called the \textit{Gysin
  map}.
It plays an important role in enumerative geometry.
Using the equivariant localization formula, the author
gave in 2017 a general formula for the Gysin map of a fiber
bundle with equivariantly formal fibers.
Equivariantly formal manifolds include all manifolds with cohomology in only even degrees such as complex projective spaces, Grassmannians, and flag manifolds as well as $G/H$, where $G$ is a compact Lie group and $H$ is a closed subgroup of maximal rank.
\tred{This article is a simplified exposition using the example of a projective bundle 
to illustrate the algorithm.}
\end{abstract}
\maketitle

\bigskip
\bigskip




\tred{To simplify the presentation, we will use homology and cohomology with real coefficients throughout the article}.
A smooth map $f\colon E \to M$ of compact oriented manifolds naturally
induces a \textit{pullback map} in cohomology.
Via Poincar\'e duality, the pushforward map in homology
$f_*\colon H_*(E) \to H_*(M)$ induces a \textit{pushforward map} in
cohomology, called the \term{Gysin map} and also denoted by $f_*$:
\[
\xymatrix{
H^k(E) \ar@{..>}[r]^-{f_*} \ar[d]_-{{\rm P.D.}}^-{\simeq} & 
H^{k-(e-m)}(M) \ar[d]^-{{\rm P.D.}}_-{\simeq} \\
H_{e-k}(E) \ar[r]_-{f_*} & H_{e-k}(M),
}
\]
where $e$ and $m$ are the dimensions of $E$ and $M$ respectively and
the vertical maps Poincar\'e duality isomorphisms.

Formulas for the Gysin map have played an important role in the theory
of determinantal varieties and in Brill-Noether theory in algebraic
geometry.
Special cases for projective bundles, Grassmann bundles, and flag
bundles have been obtained using various methods, including
combinatorial algebra and representation theory (\cite{pragacz},
\cite{fulton--pragacz}, \cite{akyildiz--carrell}, \cite{brion96}).

For example, let $V \to M$ be a smooth complex vector bundle of rank $r$ with
total Chern class
\[
  \tred{c(V) = \sum_{i=0}^r c_i(V) \in H^*(M), \text{ where } c_0(V) = 1,}
\]
and $f\colon P(V) \to M$ the associated projective bundle of $V$.
Over $P(V)$ there are the tautological subbundle $S$ and its dual
$S\dual$, \tred{the hyperplane bundle.}
Denote by $x = c_1(S\dual)$ the first Chern class of
$S\dual$, \tred{called the \term{hyperplane class}}.
Then the cohomology ring of $P(V)$ is
\[
  H^*\big( P(V) \big) = H^*(M)[x]/\big( x^r + c_1(V) x^{r-1} + \cdots
  + c_r(V) \big).
\]
Thus, the Gysin map $f_*\colon H^*\big( P(V) \big) \to H^*(M)$ is
completely described by $f_*(x^k)$ for all $k \ge 0$.
This in turn is encapsulated in the beautiful classical formula
\begin{equation} \label{e:projective}
  f_* \left( \frac{1}{1-x} \right) = \frac{1}{c(V)},
\end{equation}
or
\[
  f_* \left( \sum_{k=0}^{\infty} x^k \right) = \frac{1}{1 + c_1(V) +
    \cdots + c_r(V)}
\]
(see \cite[Eq.~4.3, p.~318]{arbarello--cornalba--griffiths--harris} for a proof).

In \cite{tu2017} it is shown that the localization formula in
equivariant cohomology provides a systematic method of calculating the
Gysin map.
With this method, prior Gysin formulas can be generalized to any fiber
bundle with equivariantly formal fibers.
The purpose of this article is to illustrate the method with one key
example,
the Gysin formula for an associated projective bundle.
In fact, this may be the best way to understand the algorithm.
In particular, we rederive formula \eqref{e:projective} for a
projective bundle.
While the classical proof of this particular case is quite short, the
method using equivariant cohomology has wider applicability.

\section{Universal Fiber Bundles}

In this section we work in the continuous category.  For a topological
group $G$, just as principal $G$-bundles have a universal bundle, so
we will show that fiber bundles with fiber $F$ and structure group $G$
also have a universal bundle.

\begin{definition}
  Let $f\colon E \to M$ be a fiber bundle with fiber $F$ and structure
  group $G$.  An \term{associated principal bundle} of $f$ is a
  principal $G$-bundle $P \to M$ such that $E \to M$ is isomorphic to
  the associated fiber bundle $P \times_G F \to M$.
  \end{definition}

\begin{proposition}
 Every fiber bundle $f\colon E \to M$ with fiber $F$ and
  structure group $G$ has an associated principal $G$-bundle.
\end{proposition}

\begin{proof}
  Let $\{ U_{\ga}\}$ be a trivializing open cover of $M$ for $E$ and
  $\{g_{\ga\gb}\colon U_{\ga} \cap U_{\gb} \to G\}$ the transition
  functions.
  Write $U_{\ga\gb} = U_{\ga} \cap U_{\gb}$.
  We construct the space $P = \big( \prod (U_{\ga} \times G)\big)/\sim$ by
  identifying
  \[
    (m, g_{\ga}) \sim \big( m, g_{\gb\ga}(m) g_{\ga}\big) \quad\text{
      for }
    \tred{(m, g_{\ga}) \in U_{\ga} \times G, \  \big( m, g_{\gb\ga}(m) g_{\ga}\big)\in U_{\gb} \times G,}
  \]
  if $m\in U_{\ga\gb}$. 
  The natural projection $\varphi\colon P \to M$, $\varphi \big([m, g_{\ga}]\big) = m$, will be a principal $G$-bundle with
  transition functions $\{g_{\ga\gb}\colon$ $U_{\ga\gb} \to G\}$.
  The fiber bundle $f\colon E \to M$ presumes an action of $G$ on $F$.
  With this action, the associated bundle $P \times_G F \to M$ is a
  fiber bundle with fiber $F$ and group $G$.
  Since it has the same transition functions $\{g_{\ga\gb}\}$ as $E$
  relative to the open cover $\{U_{\ga}\}$, the two bundles $P \times_G F$ and
  $E$ are isomorphic as fiber bundles over $M$.\qed
\end{proof}

By Milnor \cite{milnor}, every principal $G$-bundle has a universal principal $G$-bundle $\ga\colon EG \to BG$.  It follows from this that the same is true for fiber bundles with fiber $F$ and group $G$.
\tred{When a topological group acts on the left on a space $F$, the \term{homotopy quotient} $F_G$ of $F$ by $G$ is the quotient of $EG \times F$ by the diagonal action  $(e,x) \cdot g = (eg, g^{-1}x)$ for $(e,x) \in EG \times F$ and $g\in G$. It is a fiber bundle over $BG$ with fiber $F$ via $(e,x) \mapsto \ga(x) \in BG$.}

\begin{theorem} \label{t:universal}
  Let $f\colon E \to M$ be a fiber bundle with fiber $F$ and structure group
  $G$ and let $F_G$ be the homotopy quotient of $F$ by $G$.
  Then there is a map $\hu\colon M \to BG$ such that $E$ is isomorphic
  to the pullback $\hu^*(F_G)$ of the bundle $F_G\to BG$.
  \end{theorem}

  \begin{proof}
    Let $\varphi\colon P \to M$ be the associated principal $G$-bundle of the fiber
    bundle $f\colon E \to M$.
If $\hu\colon M
\to BG$ is the classifying map for the principal bundle $P$, then there is a commutative
diagram
\[
\xymatrix{
P \ar[d]_-{\varphi} \ar[r] & EG \ar[d] \\
M \ar[r]_-{\hu} & BG
}
\]
such that $P$ is isomorphic to the pullback $\hu^*(EG)$.
The map $P \to EG$ induces a map $h\colon P \times_G F \to EG \times_G
F$ and a commutative diagram

\begin{equation} \label{e:universal}
\begin{xy}
(1,10.5)*{E =};
(6,5)*{\xymatrix{
P\times_G F \ar[d]_-f \ar[r]^-h & 
EG \times_G F \ar[d]^-{\pi} \\
M \ar[r]_-{\underline{h}} & BG.}
};
(52,10)*{=F_G};
\end{xy}
\end{equation}
\vspace{.5in}

\noindent
such that $E$ is isomorphic to the pullback $\underline{h}^*(EG \times_G F)$.
\qed\end{proof}

Thus, $F_G \to BG$ can be viewed as a universal fiber bundle with fiber $F$ and
structure group $G$.

\section{Comparison of $G$- and $T$-Equivariant Cohomology}

Let $G$ be a compact connected Lie group, $T$ a maximal torus,
$N_G(T)$ the normalizer of $T$ in $G$, and $W = N_G(T)/T$ the Weyl
group of $T$ in $G$.
Suppose $G$ acts freely on the right on a topological space $X$ 
\tred{in such a way}
that $X \to X/G$ is a principal $G$-bundle.
Then the natural projection $X/T \to X/G$ is a fiber bundle with fiber
$G/T$.
The Weyl group $W$ acts on $X/T$ by
\[
  (xT) \cdot w = xwT \quad\text{ for } x \in X \text{ and } w\in W.
\]
This action induces an action on $W$ on the cohomology $H^*(X/T)$.

\begin{lemma}[{\cite[\S 3, Lemma 4]{tu2010}}] \label{l:invariants}
  With the notations above, the rational cohomology of $X/G$ is the
  ring of $W$-invariants of the rational cohomology of $X/T$:
  \[
    H^*(X/G) \simeq H^*(X/T)^W.
  \]
  \end{lemma}

  Given a left $G$-space $M$, with the diagonal action of $G$ on $X =
  EG \times M$,  this lemma gives the $G$-equivariant cohomology of
  $M$ as the ring of $W$-invariants of the $T$-equivariant cohomology of $M$:
  \begin{equation} \label{e:invariants}
    H_G^*(M) \simeq H_T^*(M)^W.
  \end{equation}

  \section{Fiber Bundles with Equivariantly Formal Fibers}

  If $X$ is a topological space on which a topological group $G$ acts,
  let $X_G = (EG \times X)/G$ be the homotopy quotient.
  Then $X_G \to BG$ is a fiber bundle with fiber $X$ \tred{\cite[Prop.~4.5, p.~33]{tu2020}}.
  The inclusion $X \hookrightarrow X_G$ of a fiber induces a
  restriction map $H_G^*(X) \to H^*(X)$.

  \begin{definition}
    A topological space $X$ with a group action by $G$ is said to be
    \term{equivariantly formal} if the restriction map $H_G^*(X) \to
    H^*(X)$ from equivariant cohomology to ordinary cohomology is
    surjective.
  \end{definition}

  We will now determine the cohomology of a fiber bundle $f\colon E
  \to M$ with equivariantly formal fiber $F$ and structure group $G$.
  As in Theorem~\ref{t:universal}, the fiber bundle $E$
  has an associated principal bundle $\varphi\colon P \to M$ such 
  that $E= P \times_G F$ and $E$  maps to \tred{the}
  universal bundle $F_G$.
  \tred{Fix a point $p_0\in P$. Then there is a commutative diagram}
  \begin{equation} \label{e:1}
\xymatrix{
  F\ \ar@{^{(}->}@<-.1ex>[r]^-{i_E}  \ar@/^1.5pc/ @{.>}[rr]^i
  &E \ar[d]_-f \ar[r]^-h & F_G \ar[d]^-{\pi_G} \\
& M \ar[r]_-{\hu} & BG,
}
\qquad
\xymatrix{
z\ \ar@{|->}@<-.1ex>[r]^-{i_E}  \ar@/^1.5pc/ @{.>}[rr]^i
  &[p_0, z]  \ar@{|->}[r]^-h & [\varphi(p_0), z] 
  }
\end{equation}
\tred{for} $z \in F$.

\begin{lemma}
  Denote by $i_E\colon F \to E$ and $i\colon F \to F_G$ the inclusions
  of $F$ as fibers of $E \to M$ and $F_G \to BG$ respectively.
  Then $i = h \comp i_E$,  where $h\colon E \to F_G$ is the
  map
  defined in \eqref{e:universal}.
\end{lemma}

\begin{proof}
  Let $\varphi\colon P\to M$ be the principal $G$-bundle associated to the fiber
  bundle $f\colon E \to M$ and $\varphi\colon P \to EG$ its
  classifying map.
  Fix a base point $p_0 \in P$.
  Then
  \[
    h\colon E := P \times_G F \to EG\times_G F :=F_G
  \]
  is given by $[p_0, z] \mapsto [\varphi(p_0), z]$ for $[p_0,z] \in P
  \times_G F$,
  $i\colon F \to F_G$ is given by $z \mapsto [\varphi(p_0), z] \in
  EG\times_GF$, and
  $i_E\colon F \to E$ is given by $z \mapsto [p_0, z] \in P \times_G
  F$.
  It follows that
  \[
    (h\comp i_E)(z) = h\big( [p_0, z]\big) = [ \varphi(p_0), z] = i(z)
    \in F_G.  \tag*{\qed}
  \]
\end{proof}

It follows from this lemma that the map $h^*\colon H_G^*(F) \to H^*(E)$ is \tred{part of the restriction map $i^* = i_E^* \comp h^*$.}
By the push-pull formula (\cite[Prop.~8.3]{borel--hirzebruch} or \cite[Lem.~1.5]{chern}), the commutative diagram \eqref{e:1} induces
a commutative diagram in cohomology
  \begin{equation} \label{e:cohomology}
  \bfig
\xymatrix{
  H^*(F) 
  &H^*(E) \ar[d]_-{f_*} \ar[l]_-{i_E^*} & H_G^*(F) \ar[l]_-{h^*}
  \ar[d]^-{\pi_{G*}} \ar@/_1.75pc/ [ll]_-{i^*}
  \\
& H^*(M) & H^*(BG). \ar[l]^-{\hu^*}
}
\efig
\end{equation}

\begin{theorem} \label{t:cohomology}
    Let $f\colon E \to M$ be a fiber bundle whose fiber $F$ is
    $G$-equivariantly formal and has finite-dimensional cohomology with
    basis $a_1, \ldots, a_r$, where $G$ is the structure group of the bundle.
    Let $b_1, \ldots, b_r \in H_G^*(F) $ be \tred{equivariant extensions of $a_1, \ldots, a_r$, respectively.}
    Then the cohomology of $E$ is the $f^*H^*(M)$-module with basis
    $h^*b_1, \ldots, h^*b_r$:
    \begin{align*}
      H^*(E) &= \big(f^*H^*(M)\big)\{ h^*b_1, \ldots, h^*b_r\}\\
             &\simeq H^*(M) \otimes_{\R} H^*(F).
    \end{align*}
  \end{theorem}

  \begin{proof}
    Since $F$ is equivariantly formal, the restriction map $i^*\colon
    H_G^*(F) \to H^*(F)$ is surjective.
    From the top row of diagram \eqref{e:cohomology}  we see that
    $h^*b_1, \ldots, h^*b_r$ will restrict to the generators $a_1,
    \ldots, a_r$ on the fiber $F$.
    By the Leray--Hirsch theorem \cite[Th.~5.11, p.~50]{bott--tu}, the cohomology
    of $E$ is as claimed.
    \qed\end{proof}

  For $b\in H_G^*(F)$, we call $h^*b \in H^*(E)$  a \term{fiber class} and an element of $f^*H^*(M) \subset H^*(E)$ a \term{basic class}.  By Theorem~\ref{t:cohomology}, the cohmology of a fiber bundle $E$ with equivariantly formal fiber $F$ is generated as a ring by basic classes and fiber classes.

  \section{Gysin formulas: an Algorithm}

  In this section we explain how to use the equivariant localization
  formula to calculate the Gysin map for a $\cinf$ fiber bundle
  $f\colon E \to M$ with equivariantly formal fibers $F$ and structure
  group a compact connected Lie group $G$.
  In order to apply Theorem~\ref{t:cohomology}, we will assume that
  the fiber has finite-dimensional cohomology.

  By Theorem~\ref{t:cohomology}, every element of $H^*(E)$ is an
  $\R$-linear combination of cohomology classes of the form
  $(f^*a)h^*b$, where $a \in H^*(M)$ and $b \in H_G^*(F)$.
  By \tred{the projection formula \cite[Prop.~6.15]{bott--tu},}
  \[
    f_*\big( (f^*a) h^*b \big) = a f_*h^*b.
  \]
  Thus, it suffices to calculate the pushforward $f_*$ of fiber
  classes $h^*b \in H^*(E)$.
  From the commutative diagram \eqref{e:cohomology},
  \begin{equation} \label{e:pushpull}
    f_* h^* b = \hu^* \pi_{G*} b.
  \end{equation}
  The problem here is that there is no localization theorem available
  for a general compact Lie group $G$ with which to calculate
  $\pi_{G*}$.

  Let $T$ be a maximal torus of the compact connected Lie group $G$.
  The projections $j\colon F_T \to F_G$ and $\ju\colon BT \to BG$ fit
  into a commutative diagram
\begin{equation} \label{e:quotient}
\bfig
\xymatrix{
F_G \ar[d]_-{\pi_G} & F_T \ar[l]_-{j} \ar[d]^-{\pi_T} \\
BG & BT. \ar[l]^-{\ju}
}
\efig
\end{equation}
By the push-pull formula, this diagram induces the second square of the commutative diagram
in cohomology
\begin{equation} \label{e:equivariant}
\bfig
\xymatrix{
H^*(E) \ar[d]_-{f_*} & H_G^*(F) \ar[l]_-{h^*} \ar[d]_-{\pi_{G*}}  \ar@{^{(}->}[r]^{j^*}& H_T^*(F)  \ar[d]^-{\pi_{T*}} \\
H^*(M) &H^*(BG) \ar[l]^-{{\hu}^*}\ar@{^{(}->}[r]_-{\ju^*} & H^*(BT). 
}
\efig
\end{equation}
The first square comes from \eqref{e:cohomology}.
\tred{From the second square, for $b \in H_G^*(F)$,
\[
  \ju^* \pi_{G*} b = \pi_{T*}(j^*b).
\]
By Lemma~\ref{l:invariants}, the horizontal maps $j^*$ and $\ju^*$ in
\eqref{e:equivariant} are inclusions.
So we may omit them and write more simply,
\[
\pi_{G*}(b) = \pi_{T*}(b),
\]
which can be evaluated using the equivariant localization formula.}
The first square of \eqref{e:equivariant} then gives 
\[
f_*(h^*b) = \hu^*(\pi_{G*} b) = \hu^*(\pi_{T*} b).
\]
Since $f^*\colon H^*(M) \to H^*(E)$ is injective, we can also write
the formula for $f_*(h^*b)$ in the form
\[
f^* f_* (h^*b) = f^* \hu^*(\pi_{G*} b) = h^* \pi_G^*( \pi_{G*} b)= h^* \pi_G^*( \pi_{T*} b).
\]
In the remaining sections we will carry out this algorithm for $f_*\big(1/(1-x)\big)$ in case $f\colon P(V) \to M$ is the associated projective
bundle of a rank 3 complex vector bundle $V \to M$.

\section{The Cohomology of a Projective Bundle and of $\C P^2$} \label{s:projective}

%

Let $V \to M$ be a $\cinf$ complex vector bundle of rank 3 over a smooth compact oriented manifold $M$ and $f\colon P(V) \to M$ its associated projective bundle.
\tred{The bundle $f\colon P(V) \to M$ is a fiber bundle with fiber $\C P^2$ and structure group $G = \Unit(3)$,
where the action of $\Unit(3)$ on $\C P^2$ is induced from the standard action of $\Unit(3)$ on $\C^3$ by left multiplication.}
Over $P(V)$ the tautological subbundle \tred{$S_{P(V)}$} of rank 1 and the tautological quotient bundle \tred{$Q_{P(V)}$} of rank 2 fit into a short exact sequence of vector bundles
\[
\tred{0\to S_{P(V)} \to f^*V \to Q_{P(V)} \to 0.}
\]
If $y, q_1, q_2$ are the Chern classes $c_1(S_{P(V)}), c_1(Q_{P(V)}), c_2(Q_{P(V)})$ respectively, then by the Whitney product formula,
\[
(1+y)(1+q_1+q_2) = f^* c(V).
\]
With $x = c_1(S\dual)= -y$, the cohomology of $P(V)$ has two descriptions:
\begin{align*}
H^*\big( P(V)\big) &= \frac{H^*(M)[x]}{\big(x^3 +c_1(V) x^2+c_2(V) x + c_3(V)\big)}
\intertext{or}
H^*\big( P(V)\big)&= \frac{H^*(M)[y, q_1, q_2]}{\big( (1+y)(1+q_1+q_2) = f^* c(V)\big)}, \quad y=-x.
\end{align*}

\tred{By Thereom~\ref{t:universal}, there is a bundle map of bundles with fiber $\C P^2$:
\[
\xymatrix{
P(V) \ar[d] \ar[r]^-h & (\C P^2)_G \ar[d] \\
M \ar[r]_-{\hu} & BG.
}
\]
Over the homotopy quotient $(\C P^2)_G$, there is a tautological line subbundle that we denote by $S_G$ (because it is the homotopy quotient of the tautological subbundle $S$ over $\C P^2$).
Since $P(V) = \hu^*\big( (\C P^2)_G \big)$,
the bundle map $h\colon P(V) \to (\C P^2)_G$ maps each fiber $\C P^2$ of $P(V) \to M$
isomorphically to a fiber $\C P^2$ of $(\C P^2)_G \to BG$.
Therefore, the tautological subbundle $S_{P(V)}$ is the pullback $h^*(S_G)$.
Denote the first Chern class of $S_G$ by $\ty$.
By the naturality of Chern classes,
\[
y=c_1(S_{P(V)})= c_1 \big( h^*(S_G)\big) = h^* c_1(S_G) = h^* \ty.
\]
Similary, if $Q_G$ denotes the tautological quotient bundle over $(\C P^2)_G$ and $\tilde{q}_i = c_i(Q_G)$,
then $q_i = h^* \tilde{q}_i$, $i= 1,2$.}

\tred{A maximal torus of $G=\Unit(3)$ is $T = \Unit(1)\times \Unit(1) \times
\Unit(1)$ embedded in $\Unit (3)$ as the subgroup of diagonal matrices:
\[
  \diag (t_1, t_2, t_3) =
  \begin{bmatrix}
    t_1 & & \\
    & t_2 & \\
    & & t_3
  \end{bmatrix}.
  \] 
The action of $G$ on $\C P^2$ induces an action of $T$ by restriction.}
Since $H_G^*(\C P^2)$ is a subring of $H_T^*(\C P^2)$, the classes
$\tilde{y}, \tilde{q}_1, \tilde{q}_2$ are also in $H_T^*(\C P^2)$.
In \cite[bottom formula, p.~204]{tu2010}, it is shown that the $T$-equivariant cohomology of $\C P^2$ is
\[
H_T^*(\C P^2)=\dfrac{\R [u_1, u_2, u_3, \ty, \tq_1, \tq_2]}
{((1+\ty)(1+\tq_1+\tq_{2}) - \prod(1+u_i))}.
\]

\section{Line Bundles on $BT$}

\tred{Let $G= \Unit(3)$ and $T= \Unit(1) \times \Unit(1) \times \Unit(1)$, a maximal torus in $G$ .}
The maps $\chi_i\colon T \to S^1$ given by $\chi_i \big( \diag (t_1, t_2,
t_3) \big)= t_i$ are group homomorphisms, called \term{characters} of the
torus $T$.
Each character $\chi_i$ can be viewed as a representation of $T$ on
$\C$ via $t\cdot v = \chi_i(t) v$.
We will denote by $\C_{\chi_i}$ the vector space $\C$ with the action
of $T$ given by the character $\chi_i$.

Let $S_{\chi_i} := ET \times _{\chi_i} \C \to BT$ be the line bundle
associated to 
the universal principal $T$-bundle $ET\to BT$ via the character $\chi_i$.
\tred{It is} the homotopy quotient $\big(
\C_{\chi_i} \big)_T$ of $\C$ with the $T$-action given by $\chi_i$.
If $u_i = c_1\big( S_{\chi_i}\big)$ is the first Chern class of the
line bundle $S_{\chi}$ on $BT$, then
\begin{align*}
  H^*(BT) &= H^*\big( B(S^1 \times S^1 \times S^1) \big) =
            H^*(BS^1 \times BS^1 \times BS^1) \\
  &= H^*(\C P^{\infty}) \otimes H^*(\C P^{\infty}) \otimes H^*(\C
    P^{\infty})\\
          &= \R [u_1] \otimes \R[u_2] \otimes \R[u_3] \\
          &= \R [u_1, u_2, u_3].
\end{align*}

\section{The Restriction Formula}

Recall our notations:  the tautological subbundle on $\C P^2$ is the
line bundle $S \to \C P^2$ and its first Chern class is $y=c_1(S)$.
As in Section~\ref{s:projective},  \tred{$\ty = c_1^G(S)$ is a} $G$-equivariant extension.
The commutative diagrams
\[
  \xymatrix{
    & F_T \ar[dd]^-j & & & H_T^*(F) \ar[dl] \\
    F \ar[ur] \ar[dr] & & \Rightarrow & H^*(F) &  \\
    & F_G & & & H_G^*(F) \ar[ul] \ar[uu]_-{j^*}
  }
\]
with $F = \C P^2$ show that $j^*\ty$ is a $T$-equivariant 
extension of $y \in H^2(\C P^2)$.
Since $j^*$ is injective, we may abuse notation and write $\ty$ for $j^*\ty$.

The action of $T$ on $\C P^2$ has three fixed points:
\[
  p_1 = (1, 0, 0), \quad p_2= (0,1,0) \quad p_3 = (0, 0, 1).
\]

\begin{theorem}[The restriction formula] \label{t:restriction}
  Let $i_{p_j}\colon  p_j \to \C P^2$ be the inclusion of the
  point $p_j$ in $\C P^2$.
  Then the restriction map $i_{p_j}^*\colon H_T^*(\C P^2) \to
  H_T^*\big(p_j\big)$ is given by
  \[
    i_{p_j}^*(\ty) = u_j.
  \]
\end{theorem}

\begin{proof}
  The fiber of the tautological subbundle $S$ at the fixed point $p_1$
  is $\C = \{ (*, 0, 0) \in \C^3 \}$.
  Hence, the action of $T$ on $S$ restricted to the fiber of $S$ at $p_1$ is $\C_{\chi_1}$.
  \tred{Similarly}, the fiber of $S$ at $p_j$ with the action of $T$ is $\C_{\chi_j}$.
  We therefore have a commutative diagram
  \[
    \xymatrix{
      \C_{\chi_j} \ar@{^{(}->}[r] \ar[d] & S \ar[d] \\
      p_j \ar@{^{(}->}[r]_-{i_{p_j}} & \C P^2 .}
    \]
    Taking homotopy quotients gives \tred{a commutative diagram of $T$-spaces and $T$-equivariant maps}
    \begin{equation} \label{e:quotient}
\begin{xy}
(-0.5,10.5)*{S_{\chi_j} = };
(6,5)*{    \xymatrix{
      \big(\C_{\chi_j}\big)_T\  \ar@{^{(}->}[r] \ar[d] & S_T \ar[d] \\
      (p_j)_T\  \ar@{^{(}->}[r] & (\C P^2)_T.}
};
(1,-5)*{BT=};
\end{xy}
\end{equation}

\bigskip

\noindent
It follows that the homotopy quotient $S_T$ of the tautological bundle $S$ restricts to the line bundle $S_{\chi_j}$ over $(p_j)_T=BT$: 
\[
i_{p_j}^* S_T = S_{\chi_j}
\] 
and
\[
  i_{p_j}^* \ty = i_{p_j}^*\big( c_1^T(S) \big) = i_{p_j}^*\big( c_1(S_T) \big)
  = c_1(i_{p_j}^* S_T) = c_1 (S_{\chi_j} ) = u_j.  \tag*{\qed}
\]
\end{proof}

\section{Equivariant Euler Class of the Normal Bundle}

\tred{We will be applying the equivariant localization formula in the following form
\cite[Eq.~(3.8), p.~9]{atiyah--bott84}.

\begin{theorem}
Suppose a torus $T$ acts on a compact oriented manifold $F$, and $\phi \in H_T^*(F)$ is a $T$-equivariant cohomology class of $F$.
Then
\[
\pi_{T*}\phi = \sum_{p \in F^T} \frac{i_p^* \phi}{e^T(\nu_p)} \in H^*(BT),
\]
where $F^T$ is the fixed point set of $T$ acting on $F$, $\nu_p$ is the normal bundle of $p$ in $F$, and $e^T$ is the equivariant Euler class.
\end{theorem}
}

To apply \tred{this} formula, we need to be able to
evaluate the equivariant Euler class $e^T(\nu_p)$ of the normal bundle
$\nu_p$ at a fixed point $p$ of the $T$-action on $\C P^2$.
\tred{The tangent bundle of the projective space $\C P^n$ can be expressed as \cite[Lemma 4.4, p.~43]{milnor--stasheff}
\[
T ( \C P^n ) \simeq \Hom(S,Q) \simeq S\dual \otimes Q.
\]}

When the fixed point $p$ is isolated, the normal bundle $\nu_p$ is
simply the tangent space
\[
  \nu_p = T_p(\C P^2) = \Hom (S_p, Q_p) = S_p\dual \otimes Q_p,
\]
where $S$ is the tautological subbundle of rank 1 and $Q$ the tautological quotient
bundle of rank 2 on $\C P^2$, and $S_p, Q_p$ their fibers at $p$.
Then
\begin{align*}
  e^T(\nu_p) &= e^T(S_p\dual \otimes Q_p) = e\big( (S_p\dual)_T \otimes
               (Q_p)_T \big) \\
             &= c_2 \big( (S_p\dual)_T \otimes (Q_p)_T\big) \quad \text{(Euler class = top Chern class)}.
\end{align*}

As noted in the preceding section, at the fixed point $p_1$, the fiber
$S_{p_1}$ is the representation $\C_{\chi_1}$.
Similarly, the fiber $Q_{p_1}$ is the representation $\C_{\chi_2}
\oplus \C_{\chi_3}$.
Hence,
\[
  (S_{p_1}\dual)_T = (\C_{-\chi_1})_T = S_{\chi_1}\dual \quad \text{
    and }
 (Q_{p_1})_T = (\C_{\chi_2} \oplus \C_{\chi_3})_T = S_{\chi_2} \oplus
 S_{\chi_3},
\]
so that the equivariant Euler class of the normal bundle $\nu_{p_1}$ at $p_1$ is
\begin{align*}
  e^T(\nu_{p_1}) &= c_2\big( S_{\chi_1}\dual \otimes (S_{\chi_2}
                   \oplus S_{\chi_3})\big)\\
                 &= c_2\big( (S_{\chi_1}\dual \otimes S_{\chi_2}) \oplus
                   (S_{\chi_1}\dual \otimes S_{\chi_3})\big)\\
        &=c_1  (S_{\chi_1}\dual \otimes S_{\chi_2}) c_1
          (S_{\chi_1}\dual \otimes S_{\chi_3}) \qquad (\text{Whitney
          product formula})\\
                 &= (u_2 - u_1) (u_3 - u_1).
\end{align*}
Similar formulas hold at $p_2$ and $p_3$:
\[
  e^T(\nu_{p_2}) = (u_1 - u_2) (u_3 - u_2), \qquad
  e^T(\nu_{p_3}) = (u_1 - u_3) (u_2 - u_3).
\]

\section{Evaluation of the Equivariant Localization Formula}

We can finally evaluate the pushforward \eqref{e:projective} using the
equivariant localization formula.
For a complex vector bundle $V$ of rank 3 over a compact oriented manifold $M$,
the structure group can be taken to be the unitary group $U(3)$.
It has maximal torus $T = U(1) \times U(1) \times U(1)$ represented by the $3 \times 3$ diagonal unitary matrices.
The group $U(3)$ acts on $\C^3$ by left multiplication.
This induced action on $\C P^2$ has three fixed points
\[
p_1 = [1, 0, 0], \quad p_2=[0,1,0], \quad p_3 = [0,0,1].
\]

At the fixed point $p_j$, by the restriction formula
(Theorem~\ref{t:restriction}),
\[
  i_{p_j}^*(\ty) = u_j.
\]
Hence,
\[
  i_{p_j}^*\left(\frac{1}{1+\ty}\right) = \frac{1}{1+u_j}.
\]
The equivariant Euler classes $e^T(\nu_{p_j})$ at $p_j$ for $j= 1, 2,
3$ are
\[
  (u_2 - u_1) (u_3 - u_1), \quad (u_1 - u_2) (u_3 - u_2), \quad
  (u_1 - u_3) (u_2 - u_3),
\]
respectively.

By the equivariant localization formula for a $T$-action,
\begin{align} \label{e:pushforward}
  \pi_{T*}\left(\frac{1}{1+\ty}\right) 
  &= \sum_{j=1}^3 \frac{i_{p_j}^*\left(\frac{1}{1+\ty}\right)}{e^T(\nu_{p_j})} \notag\\
  &= \frac{1/(1+u_1)}{(u_2 - u_1)(u_3 - u_1)}
  +\frac{1/(1+u_2)}{(u_1 - u_2)(u_3 - u_2)}
  + \frac{1/(1+u_3)}{(u_1 - u_3)(u_2 - u_3)} \notag\\
  &= \frac{1}{(1+u_1)(1+u_2)(1+u_3)} \qquad (\text{\tred{by hand or} using
      Mathematica})  
\end{align}
This element is in $H^*(BT) = \R[u_1, u_2, u_3]$.
Since it is invariant under the Weyl group $W_G(T) = S_3=$ symmetric group on three letters, it actually
lies in 
\tred{\[
H^*(BG) = H^*(BT)^{S_3} = \R [u_1, u_2, u_3]^{S_3}.
\]}

Using the commutative diagram~\eqref{e:equivariant}, with $E= P(V)$, $F= \C P^2$, $G= U(3)$, and $T= U(1) \times U(1) \times U(1)$, we have
\begin{alignat*}{2}
f_* \frac{1}{1-x} &= f_* \frac{1}{1+y} = f_*h^* \frac{1}{1+\ty}\\
&= \hu^* \pi_{G*} \frac{1}{1+\ty} &\quad&\text{(commutativity of first square in \eqref{e:equivariant})}\\
&= \hu^* \ju^*\pi_{G*} \frac{1}{1+\ty} &\quad&\text{(because $\ju^*$ is an inclusion)}\\
&= \hu^* \pi_{T*}j^* \frac{1}{1+\ty} &\quad&\text{(commutativity of second square in \eqref{e:equivariant})}\\
&= \hu^* \pi_{T*} \frac{1}{1+\ty} &\quad&\text{(because $j^*$ is an inclusion)}\\
&= \hu^* \frac{1}{\prod_{j=1}^3 (1+u_j)} &\quad&(\text{by equation } \eqref{e:pushforward}).
\end{alignat*}
Thus,
\begin{alignat*}{2}
f^*f_* \frac{1}{1-x} &= f^*\hu^* \frac{1}{\prod_{j=1}^3 (1+u_j)}\\
&=h^* \pi_G^* \frac{1}{\prod_{j=1}^3 (1+u_j)} &\quad&\text{(from the commutative diagram~\eqref{e:1})}\\
&=h^*\frac{1}{\prod_{j=1}^3 (1+u_j)} &\quad&\left(\prod_{j=1}^3 (1+u_j) \text{ is in the coefficient ring \tred{$H^*(BG$)}} \right)\\
&=h^*\frac{1}{(1+\ty)(1+\tq_1+\tq_2)}&\quad& \left(\text{In $H_G^*(\C P^2)$, $(1+\ty)(1+\tq_1+\tq_2)-\prod_{j=1}^3 (1+u_j) = 0$}\right)\\
\intertext{\ }\vspace{-40pt}
&= \frac{1}{(1+y)(1+q_1+q_2)}\\
&= \frac{1}{f^*c(V)}.
\end{alignat*}
Since $f^*\colon H^*(M) \to H^*(P(V))$ is injective, we finally obtain
\[
f_* \frac{1}{1-x}  = \frac{1}{c(V)}.
\]

%

\end{document}